\documentclass[12pt]{extarticle}
\usepackage[utf8]{inputenc}
\usepackage[english]{babel}
\usepackage{indentfirst}
\usepackage{amsmath}
\usepackage{amsfonts}
\usepackage{amssymb}
\usepackage{amsthm}
\usepackage{hyperref}
\usepackage[left=2.5cm,right=2.5cm,
top=2.5cm,bottom=2.5cm,bindingoffset=0cm]{geometry}
\usepackage{tikz}
\usetikzlibrary{decorations.markings}
\usepackage{caption}
\usepackage{multirow}
\usepackage{textcomp}
\usepackage{array}

\newtheorem{theorem}{Theorem}
\newtheorem{lemma}{Lemma}
\newtheorem{corollary}{Corollary}
\newtheorem{remark}{Remark}
\newtheorem{conjecture}{Conjecture}

\let\oldbibliography\thebibliography
\renewcommand{\thebibliography}[1]{
	\oldbibliography{#1}
	\setlength{\itemsep}{-0.5pt}}

\renewcommand{\P}{\mathbb{P}}
\renewcommand{\O}{\mathcal{O}}
\renewcommand{\Re}{\mathbb{R}}
\renewcommand{\L}{\mathcal{L}}

\newcommand{\simTo}{\xrightarrow{\smash{\raisebox{-0.7ex}
			{\ensuremath{\scriptstyle\!\!\sim}}}}}
\newcommand{\Supp}{\mathrm{Supp}}
\newcommand{\rmA}{\mathrm{A}}
\newcommand{\N}{\mathbb{N}}

\newcommand{\Z}{\mathbb{Z}}
\newcommand{\bfQ}{\mathbf{Q}}
\newcommand{\bfP}{\mathbf{P}}
\newcommand{\bfL}{\mathbf{L}}
\newcommand{\Aut}{\mathrm{Aut}}
\newcommand{\Tr}{\mathrm{Tr}}
\newcommand{\G}{\mathbb{G}}
\newcommand{\Pic}{\mathrm{Pic}}
\newcommand{\GL}{\mathrm{GL}}
\newcommand{\EQS}{\mathcal{E}}
\newcommand{\D}{\mathcal{D}}
\newcommand{\C}{\mathcal{C}}
\newcommand{\BCH}{\mathrm{BCH}}
\newcommand{\cd}{\!\cdot\!}

\newcommand{\F}[1]{\mathbb{F}_{\!#1}}
\newcommand{\A}[1]{\mathbb{A}^{\!#1}}
\newcommand{\h}[1]{\widehat{#1}}
\newcommand{\wt}[1]{\widetilde{#1}}
\newcommand{\str}[1]{\stackrel{\approx}{#1}}
\newcommand{\inv}[1]{#1^{-\!1}}

\begin{document}	
	
	\setlength{\medmuskip}{1\medmuskip}
	\setlength{\thickmuskip}{1\thickmuskip}	
	
	\tikzset{->-/.style={decoration={
				markings,
				mark=at position #1 with {\arrow{latex}}},postaction={decorate}}}
	\tikzset{-<-/.style={decoration={
				markings,
				mark=at position #1 with {\arrow{latex reversed}}},postaction={decorate}}}	

	\begin{center}
		
		{\Large\bf Non-split toric BCH codes on singular del Pezzo surfaces}
		
		\vspace{0.5cm}
		
		{\large Dmitrii Koshelev} \footnote{
			 Web page: https://www.researchgate.net/profile/Dimitri\_Koshelev\\ 
			 \phantom{\hspace{0.65cm}} Email: dishport@yandex.ru\\
			 The work was supported in part by the Simons Foundation.
		}
		
		\vspace{0.25cm}
		
		Versailles Laboratory of Mathematics, Versailles Saint-Quentin-en-Yvelines University \\
		Center for Research and Advanced Development, Infotecs \\
		Algebra and Number Theory Laboratory, Institute for Information Transmission Problems \\
		
		\vspace{0.5cm}		
		
	\end{center}

{\bf \abstractname.} In the article we construct low-rate non-split toric $q$-ary codes on some singular surfaces. More precisely, we consider non-split toric cubic and quartic del Pezzo surfaces, whose singular points are $\F{q}$-conjugate. Our codes turn out to be BCH ones with sufficiently large minimum distance $d$. Indeed, we prove that $d - d^* \geqslant q - \lfloor 2\sqrt{q} \rfloor - 1$, where $d^*$ is the designed minimum distance. In other words, we significantly improve upon BCH bound. On the other hand, the defect of the Griesmer bound for the new codes is $\leqslant \lfloor 2\sqrt{q} \rfloor - 1$, which also seems to be quite good. It is worth noting that to better estimate $d$ we actively use the theory of elliptic curves over finite fields.


\vspace{0.25cm}


\section*{Introduction}
\addcontentsline{toc}{section}{Introduction}

This article continues our first one \cite{MyArticle2019} about {\it non-split toric codes}, i.e., algebraic geometry (AG) codes \cite{TsfasmanVladutNogin} on non-split toric varieties \cite{CoxLittleSchenck} over a finite field $\F{q}$. It is wonderful circumstance that most of these codes are (simple-root) cyclic \cite[Chapter 7]{MacWilliamsSloane}. Therefore they have more chances to be used in practice than other algebraic geometry codes on high-dimensional varieties. In \cite{MyArticle2019} we assume everywhere that toric varieties are smooth, however there are no any obstacles to consider non-split toric codes on singular ones.

There is the well known classification of toric (possibly singular) {\it del Pezzo surfaces} \cite{CorayTsfasman}. They bijectively correspond (up to an equivalence) to so-called {\it reflexive polygons} \cite[\S 8.3]{CoxLittleSchenck}. There are exactly $16$ such polygons \cite[Theorem 8.3.7]{CoxLittleSchenck}, but only $5$ of them (see Figure \ref{symmetricReflexivePolygons} and Table \ref{toricSurfaces}) are quite \grqq symmetric'', i.e., have an integral action of order greater than $2$. The last condition seems to be necessary for constructing good non-split toric codes.

Non-split toric codes $C_6, C_8, C_9$ (Tables \ref{ToricCodesOnDelPezzoSurfaces} and \ref{ParityCheckPolynomials}) associated with such smooth polygons (i.e., $Pol_6, Pol_8, Pol_9$) have already been considered, for example, in \cite[\S 2.3]{MyArticle2019} (also see \cite[\S 4.2]{Blache}), \cite[Proposition 4.7]{CouvreurDuursma}, and \cite[\S 2]{Lachaud1990} respectively. The other polygons $Pol_3, Pol_4$ correspond to some singular cubic (\S \ref{surfaceS3}) and quartic (\S \ref{surfaceS4}) del Pezzo surfaces respectively. As far as we know, algebraic geometry codes $C_3, C_4$ (Tables \ref{ToricCodesOnDelPezzoSurfaces} and \ref{ParityCheckPolynomials}) on the given surfaces have not been studied yet. However, AG codes on some smooth cubic and quartic del Pezzo surfaces are described in \cite[\S 6]{Blache}, \cite{Boguslavsky}, \cite[\S 4.1, \S 5.1]{LittleSchenck}, and \cite{Zarzar}. 

{\bf Acknowledgements.} The author expresses his deep gratitude to his scientific advisor M. Tsfasman and thanks A. Perepechko, A. Trepalin, and K. Shramov for their help and useful comments.

\tableofcontents

\section{Toric del Pezzo surfaces and reflexive polygons}

Let $\F{q}$ be a finite field of characteristic $p$. Consider a toric (possibly singular) {\it del Pezzo surface} $S$ over $\F{q}$, i.e., a toric one, whose anticanonical divisor $-K_S$ is an ample Cartier divisor. Let $\varphi_{min}\!: S^\prime \to S$ be the minimal resolution of singularities. The surface $S^\prime$ is a so-called {\it weak} (or {\it generalized}) {\it del Pezzo surface}. The self-intersection $K^2_S$ is said to be {\it degree} of $S$ (or $S^\prime$). Besides, the {\it Fano index} of $S$ is the maximal number $i \in \N$ such that $K_S \sim iH$ for some Cartier divisor $H$ on $S$, which can be taken over $\F{q}$. The theory of (not necessarily toric) del Pezzo surfaces (with more focus on $K^2_S = 3, 4$) can be found, for example, in \cite{CorayTsfasman}.

\begin{lemma}[\mbox{\cite[Proposition 11.2.8]{CoxLittleSchenck}, \cite[Proposition 0.6]{CorayTsfasman}, \cite[Figure 1]{Derenthal}}]
	\
	\begin{enumerate}
		\item $-K_S$ is very ample, $\dim|{-}K_S| = K^2_S$, and $3 \leqslant K^2_S \leqslant 9;$ 
		\item The surface $S$ may only have singularities of the types $\rmA_1, \rmA_2, \rmA_3$ {\rm\cite[Example 10.1.5]{CoxLittleSchenck};}
		\item $\varphi_{min}$ is a crepant morphism, i.e., $K_{S^\prime} := \varphi_{min}^*(K_S)$ is a canonical divisor.
	\end{enumerate}	
\end{lemma}

A lattice convex polygon $P \subset \Re^2$ is said to be {\it reflexive} (or {\it Gorenstein}) if $O := (0,0)$ is its internal point and the dual (convex) polygon $P^\circ$ is also lattice. In this case, $P^\circ$ is obviously reflexive.

\begin{lemma}[\mbox{\cite[Exercise 2.3.5.a, Definition 2.3.12, Theorem 10.5.10]{CoxLittleSchenck}}]
	
	If $P$ is reflexive, then
	\begin{enumerate}
		\item $O$ is the unique internal point of $P;$
		\item All vertices of $P$ are ray generators of the normal fan of $P^\circ;$
		\item $|P \cap \Z^2| + |P^\circ \cap \Z^2| = 14$.
	\end{enumerate}
\end{lemma}

\begin{theorem}[\mbox{\cite[Theorems 6.2.1, 8.3.4]{CoxLittleSchenck}}]
	The maps 
	$$
	P \mapsto (S_P, D_P) \quad \mbox{\rm \cite[\S 2.3, \S 4.2]{CoxLittleSchenck}},\qquad 
	(S, {-}K_S) \mapsto P_{\!{-}K_S} \quad \mbox{\rm \cite[\S 4.3]{CoxLittleSchenck}}
	$$ 
	are inverse to each other between reflexive polygons (up to an equivalence) and toric (possibly singular) del Pezzo surfaces provided with the anticanonical divisor that is the sum of all prime torus-invariant divisors.
\end{theorem}

\begin{theorem}[\mbox{\cite[Theorem 8.3.7]{CoxLittleSchenck}, \cite[Figure 1]{Derenthal}}]
	Up to an equivalence (isomorphism) there are exactly $16$ reflexive polygons (split toric del Pezzo surfaces).
\end{theorem}

It is immediately checked that all reflexive polygons having an action of order greater than $2$ are represented in Figure \ref{symmetricReflexivePolygons}. For a polygon $Pol_i$ the subscript $i$ is the amount of integral points on its boundary. In turn, the superscript $t$ denotes the transposition operation of $\varPhi_{i}$ as a matrix (of order $i$) in $\GL(2,\Z)$. The corresponding non-split toric del Pezzo surfaces (with their Fano index) are contained in Table \ref{toricSurfaces}. We recall that the action $\varPhi_{i}$ (or $\varPhi_{i}^{\,t}$) complies with the Frobenius action on toric invariant curves and points of the surface. Finally, it is notable that all the five surfaces have Picard $\F{q}$-number $1$. 

\begin{figure}
	\centering 
	
	\begin{minipage}{0.49\linewidth}
		\begin{minipage}{0.49\linewidth}
			\centering
			
			\begin{tikzpicture}
			\draw[step = 1, style = help lines, densely dotted] (-1.5, -1.5) grid (1.5, 1.5);
			
			\filldraw[very thick, fill = gray!10!white] (1, 0) -- (0, 1) -- (-1, -1) -- (1, 0);
			
			\begin{scope}[thin, gray]
			\draw[->] (-1.5, 0) -- (1.5, 0) node[above] {\hspace{-0.5cm}$1$};
			\draw[->] (0, -1.5) -- (0, 1.5) node[right] {\raisebox{-0.7cm}{$1$}};
			\end{scope}
			
			\fill (0, 0) node[above] {\hspace{0.5cm}$O$} circle (0.04);	
			
			\begin{scope}[red, densely dashed]
			\draw[->-=.7] (0, 1) arc (90: 0: 1);		
			\draw[->-=.7] (1, 0) .. controls (0.5, -1.25) and (-0.5, -1.25) .. (-1, -1);
			\draw[->-=.7] (-1, -1) .. controls (-1.25, -0.5) and (-1.25, 0.5) .. (0, 1);		
			\end{scope}		
			\end{tikzpicture}
			\caption*{($Pol_3$, $\varPhi_3^{\,t}$)}
			
		\end{minipage}
		\begin{minipage}{0.49\linewidth}
			\centering
			
			\begin{tikzpicture}
			\draw[step = 1, style = help lines, densely dotted] (-1.5, -1.5) grid (1.5, 1.5);
			
			\filldraw[very thick, fill = gray!10!white] (1, 0) -- (0, 1) -- (-1, 0) -- (0, -1) -- (1, 0);
			
			\begin{scope}[thin, gray]
			\draw[->] (-1.5, 0) -- (1.5, 0) node[above] {\hspace{-0.5cm}$1$};
			\draw[->] (0, -1.5) -- (0, 1.5) node[right] {\raisebox{-0.7cm}{$1$}};
			\end{scope}
			
			\fill (0, 0) node[below] {\hspace{0.5cm}$O$} circle (0.04);	
			
			\begin{scope}[red, densely dashed]
			\draw[->-=.7] (0, 1) arc (90: 0: 1);
			\draw[->-=.7] (1, 0) arc (0: -90: 1);
			\draw[->-=.7] (0, -1) arc (-90: -180: 1);
			\draw[->-=.7] (-1, 0) arc (-180: -270: 1);
			\end{scope}		
			\end{tikzpicture}
			\caption*{($Pol_4$, $\varPhi_{4}^{\,t}$)}	
			
		\end{minipage}
		
		\vspace{0.5cm}
		
		\begin{minipage}{0.49\linewidth}
			\centering
			
			\begin{tikzpicture}
			\draw[step = 1, style = help lines, densely dotted] (-1.5, -1.5) grid (1.5, 1.5);
			\filldraw[very thick, fill = gray!10!white] (1, 0) -- (0, 1) -- (-1, 1) -- (-1, 0) -- (0, -1) -- (1, -1) -- (1, 0) -- (1, 0);
			
			\begin{scope}[thin, gray]
			\draw[->] (-1.5, 0) -- (1.5, 0) node[above] {\hspace{-0.5cm}$1$};
			\draw[->] (0, -1.5) -- (0, 1.5) node[right] {\raisebox{-0.7cm}{$1$}};
			\end{scope}
			
			\begin{scope}[red, densely dashed]
			\draw[->-=.7] (1, 0) arc (0: 90: 1);
			\draw[->-=.7] (0, 1) .. controls (-0.25, 1.35) and (-0.75, 1.35) .. (-1, 1);
			\draw[->-=.7] (-1, 1) .. controls (-1.35, 0.75) and (-1.35, 0.25) .. (-1, 0);
			\draw[->-=.7] (-1, 0) arc (180: 270: 1);
			\draw[->-=.7] (0, -1) .. controls (0.25, -1.35) and (0.75, -1.35) .. (1, -1);
			\draw[->-=.7] (1, -1) .. controls (1.35, -0.75) and (1.35, -0.25) .. (1, 0);
			\end{scope}	
			
			\fill (0, 0) node[below] {\hspace{0.5cm}$O$} circle (0.04);
			\end{tikzpicture}
			\caption*{($Pol_6$, $\varPhi_6$)}
			
		\end{minipage}
		\begin{minipage}{0.49\linewidth}
			\centering
			
			\begin{tikzpicture}
			\draw[step = 1, style = help lines, densely dotted] (-1.5, -1.5) grid (1.5, 1.5);
			
			\filldraw[very thick, fill = gray!10!white] (1, -1) -- (1, 1) -- (-1, 1) -- (-1, -1) -- (1, -1);
			
			\begin{scope}[thin, gray]
			\draw[->] (-1.5, 0) -- (1.5, 0) node[above] {\hspace{-0.5cm}$1$};
			\draw[->] (0, -1.5) -- (0, 1.5) node[right] {\raisebox{-0.7cm}{$1$}};
			\end{scope}
			
			\fill (0, 0) node[below] {\hspace{0.5cm}$O$} circle (0.04);	
			
			\begin{scope}[red, densely dashed]
			\draw[->-=.8] (1, 1) .. controls (0.5, 1.5) and (-0.5, 1.5) .. (-1, 1);
			\draw[->-=.8] (-1, 1) .. controls (-1.5, 0.5) and (-1.5, -0.5) .. (-1, -1);
			\draw[->-=.8] (-1, -1) .. controls (-0.5, -1.5) and (0.5, -1.5) .. (1, -1);
			\draw[->-=.8] (1, -1) .. controls (1.5, -0.5) and (1.5, 0.5) .. (1, 1);
			\draw[->-=.7] (1, 0) -- (0, 1);
			\draw[->-=.7] (0, 1) -- (-1, 0);
			\draw[->-=.7] (-1, 0) -- (0, -1);
			\draw[->-=.7] (0, -1) -- (1, 0);
			\end{scope}		
			\end{tikzpicture}
			\caption*{($Pol_8$, $\varPhi_{4}$)}	
			
		\end{minipage}	
	\end{minipage}	
	\begin{minipage}{0.49\linewidth}
		\centering
		
		\begin{tikzpicture}
		
		\begin{scope}[thin, gray]
		\draw[->] (-1.5, 0) -- (2.5, 0) node[above] {\hspace{-0.5cm}$1$};
		\draw[->] (0, -1.5) -- (0, 2.5) node[right] {\raisebox{-0.7cm}{$1$}};
		\end{scope}
		\filldraw[very thick, fill = gray!10!white] (-1, 2) -- (2, -1) -- (-1, -1) -- (-1, 2);
		\draw[step = 1, style = help lines, densely dotted] (-1.5, -1.5) grid (2.5, 2.5);
		
		\begin{scope}[red, densely dashed]
		\draw[->-=.7] (2, -1) arc (0: 90: 3);
		\draw[->-=.6] (-1, 2) .. controls (-1.5, 1) and (-1.5, 0) .. (-1, -1);
		\draw[->-=.6] (-1, -1) .. controls (0, -1.5) and (1, -1.5) .. (2, -1);
		\draw[->-=.55] (1, 0) -- (-1, 1);
		\draw[->-=.55] (-1, 1) -- (0, -1);
		\draw[->-=.55] (0, -1) -- (1, 0);
		\end{scope}	
		
		\begin{scope}[blue, densely dashed]
		\draw[->-=.55] (0, 1) -- (-1, 0);
		\draw[->-=.55] (-1, 0) -- (1, -1);
		\draw[->-=.55] (1, -1) -- (0, 1);
		\end{scope}		
		
		\fill (0, 0) node[below] {\hspace{0.5cm}$O$} circle (0.04);
		\end{tikzpicture}
		\caption*{($Pol_9$, $\varPhi_3$)}
	\end{minipage}	
	
	\caption{Reflexive polygons having an action (from \cite[Theorem 8]{MyArticle2019}) of order greater than $2$}
	\label{symmetricReflexivePolygons}
\end{figure}	

\begin{table}[h]
	\centering
	\begin{tabular}{c|c|l|c|l}
		\textnumero & (polygon, action) & toric surface & Fano index & \multicolumn{1}{c}{(polygon$^\circ$, action$^{t}$)} \\ \hline\hline 
		
		1 & ($Pol_3$, $\varPhi_3^{\,t}$) & $S_3$ (\S \ref{surfaceS3}) & \multirow{3}{*}{$1$} & 5 \\ \cline{1-3}\cline{5-5}
		2 & ($Pol_4$, $\varPhi_{4}^{\,t}$) & $S_4$ (\S \ref{surfaceS4}) & & 4 \\ \cline{1-3}\cline{5-5} 
		3 & ($Pol_6$, $\varPhi_6$) & $\D_6$ \cite[\S 2.5]{MyArticle2019} & & 3 (up to an equivalence) \\ \hline 
		4 & ($Pol_8$, $\varPhi_{4}$) & $\EQS$ \cite[\S 2.4]{MyArticle2019} & $2$ & 2 \\ \hline 
		5 & ($Pol_9$, $\varPhi_3$) & $\P^2$ & $3$ & 1
		
	\end{tabular}
	\caption{Toric del Pezzo surfaces with respect to the tori $T_3, T_4, T_6$.}
	\label{toricSurfaces}
\end{table}

\begin{lemma}[\mbox{\cite[Proposition 4.2.5, Exercises 4.3.2, 10.5.7.b]{CoxLittleSchenck}}] \label{toricProperties}
	We have:
	\begin{enumerate} 
		\item \label{Fanoindex} For $K^2_S \leqslant 7$ the Fano index of $S$ is equal to $1;$
		\item Any smooth irreducible curve from $|{-}K_S|$ is elliptic;
		\item $\Pic(S)$ is a free abelian group.
	\end{enumerate}
\end{lemma}

\subsection{Toric (singular) cubic surface in $\P^3$} \label{surfaceS3}

Choose an element $\alpha \in \F{q^3} \setminus \F{q}$ and consider the so-called {\it norm cubic $\F{q}$-surface} 
$$
S_3\!: X_0\!\cdot\! X_1\!\cdot\! X_2 = x_3^3 \quad \subset \quad \P^3_{(x_0:x_1:x_2:x_3)} \qquad \textrm{\cite[Example 1.3.10]{BatyrevTschinkel}},
$$
where 
$$
X_0 := x_0 + \alpha x_1 + \alpha^2 x_2,\qquad
X_1 := x_0 + \alpha^q x_1 + \alpha^{2q} x_2,\qquad
X_2 := x_0 + \alpha^{q^2} x_1 + \alpha^{2q^2} x_2.
$$
For $i \in \Z/3$ let
$$
\begin{array}{ll}
\h{L}_i \!: X_i = x_3 = 0,\qquad\qquad & \h{P}_i := \h{L}_{i+1} \cap \h{L}_{i+2}, \\
L_i := pr(\h{L}_i),\qquad\qquad & P_i := pr(\h{P}_i) = L_{i+1} \cap L_{i+2},
\end{array}
$$
where $pr\!: S_3 \to \P^2_{(x_0:x_1:x_2)}$ is the well-defined projection of degree $3$. Finally, let 
$$
\h{\bfL}_3 := \sum_{i=0}^2\h{L}_i, \qquad \bfL_3 := \sum_{i=0}^2 L_i, \qquad \mathrm{and} \qquad \bfP_{\!3} := \{P_i\}_{i=0}^2.
$$ 

\begin{remark}
	The surface $S_3$ is toric with respect to the torus \mbox{$T_3 \simeq S_3 \setminus \{x_3 = 0\}$} $($see {\rm\cite[Theorem 8]{MyArticle2019})} and the lines $\widehat{L}_i$ $($resp. $\widehat{P}_i)$ are the unique $T_3$-invariant curves $($resp. points$)$ on $S_3$. Moreover, they are $\F{q}$-conjugate.
\end{remark}

\begin{lemma}[\mbox{\cite[Table 7]{Derenthal}}]
	We have:
	\begin{enumerate} 
		\item The points $\widehat{P}_i$ are the unique singularities on $S_3$ $($of type $\rmA_2);$
		\item $\varphi_{min}\!: S_3^\prime \to S_3$ is the simultaneous blowing up at them;
		\item $\widehat{L}_i$ are the unique lines on $S_3$. 
	\end{enumerate}	
\end{lemma}

\begin{theorem}[\mbox{\cite[Exercise 8.3.8.c]{CoxLittleSchenck}, \cite[Example 0.7.b]{CorayTsfasman}, \cite[Table 7]{Derenthal}, \cite[Example 1.3.10]{BatyrevTschinkel}}] \label{theSurfaceS3}
	\ 
	\begin{enumerate}
		\item $S_3$ is the unique $($up to an $\F{q}$-isomorphism$)$ toric del Pezzo surface of degree $3$ with respect to the torus $T_3;$
		\item $S_3$ is the non-split toric surface associated with the pair $(Pol_3, \varPhi_3^{\,t});$
		\item \label{statement3S3} $S_3$ is the so-called fake projective plane {\rm\cite[Example 1.2]{Kasprzyk}}, i.e., the quotient $\P^2/\sigma$ under a transformation $\sigma \in \mathrm{PGL}(3, \F{q})$, whose fixed point set is $\bfP_{\!3};$ 
		\item \label{statement4S3} $S_3^\prime$ is the blowing up of the del Pezzo surface $\D_3$ of degree $6$ $($see {\rm \cite[\S 2.5]{MyArticle2019})} at one of the two triples 
		$$
		\bfQ_3 = \{Q_0, Q_1, Q_2\},\qquad \bfQ_3^\prime = \{Q_0^\prime, Q_1^\prime, Q_2^\prime\}
		$$ 
		of $\F{q}$-conjugate $T_3$-invariant points.
	\end{enumerate}	
\end{theorem}

\begin{proof}
	All statements can be found in the references, except that the action $\varPhi_3^{\,t}$ on the polygon $Pol_3$ is the only one of order $3$ \big(up to a conjugation in $\Aut(Pol_3)$\big). This fact is necessary in order to correctly pass from the split torus case (in those references) to that of $T_3$.
\end{proof} 

From \cite[Theorem 14]{MyArticle2019} or one of Statements \ref{statement3S3}, \ref{statement4S3} of Theorem \ref{theSurfaceS3} it follows that the Picard $\F{q}$-number of $S_3$ is equal to $1$. Since the Fano index of $S_3$ is also $1$, we obtain

\begin{lemma} The Picard $\F{q}$-group of the surface $S_3$ is equal to
	$$
	\Pic(S_3) = \langle{-}K_{S_3}\rangle \simeq \Z.
	$$
\end{lemma} 

For the sake of definiteness, we choose the triple $\bfQ_3$ and thus we deal with the diagram
$$
S_3 \stackrel{ \raisebox{0.2cm}{$\varphi_{min}$} }{\longleftarrow} S_3^\prime \stackrel{ \raisebox{0.2cm}{$bl_{\bfQ_3}$} }{\longrightarrow} \D_3 \stackrel{ \raisebox{0.2cm}{$bl_{\bfP_{\!3}}$} }{\longrightarrow} \P^2,
$$
where $bl_{\bfQ_3}$, $bl_{\bfP_{\!3}}$ are the blowing up maps at $\bfQ_3$, $\bfP_{\!3}$ respectively. Besides, let 
$$
\varphi := bl_{\bfP_{\!3}} \circ bl_{\bfQ_3} \circ \inv{\varphi_{min}} \qquad \varphi\!: S_3 \dashrightarrow \P^2.
$$

\begin{corollary} \label{AnticanLinearSystemS3}
	The anticanonical linear system of $S_3$ is equal to
	$$
	|{-}K_{S_3}| = \varphi^*(\L) - 2\h{\bfL}_3, \qquad \textrm{where} \qquad \L := |\bfL_3 - \bfP_{\!3} - \bfQ_3|
	$$
	is the (incomplete) linear system of all (possibly reducible or singular) $\F{q}$-cubics $C \subset \P^2$ passing through $\bfP_{\!3}$ such that $L_i$ is a tangent of $C$ at $P_{i+1}$ $($resp. $P_{i+2}$ for the triple $\bfQ_3^\prime)$.
\end{corollary}

\noindent For more clarity on what is going on, see Figures \ref{LinesLionP2AndCurveCFromL}, \ref{CurvesOnSurfaceD3}, where arrows denote the Frobenius action. In the second figure $E_{P_i}$ are the exceptional curves associated with the points $P_i$ and $\widetilde{L}_i$ (resp. $\widetilde{C}$) are the proper preimages of $L_i$ (resp. $C \neq \bfL_3$) with respect to $bl_{\bfP_{\!3}}$. As usual, we also use the notations
$$
E_{\bfP_{\!3}} := \sum_{i=0}^2E_{P_i} \qquad \mathrm{and} \qquad 
\wt{\bfL}_3 := \sum_{i=0}^2\wt{L}_i. 
$$

\begin{proof}
	Let us freely use known identities for direct and inverse images of (possibly incomplete) linear systems on algebraic surfaces (see, e.g., \cite[\S II.5-6, \S IV.2]{Hartshorne}). First,
	$$
	bl_{\bfP_{\!3}}^*\big| \bfL_3 - \bfP_{\!3} \big| = 
	\big| bl_{\bfP_{\!3}}^*(\bfL_3) - E_{\bfP_{\!3}} \big| + E_{\bfP_{\!3}} = 
	\big| \widetilde{\bfL}_3 + E_{\bfP_{\!3}} \big| + E_{\bfP_{\!3}}.
	$$
	Therefore
	$$
	\L^* := bl_{\bfP_{\!3}}^*(\L) = \wt{\L} + E_{\bfP_{\!3}}, \qquad \mathrm{where} \qquad 
	\wt{\L} := \big| \widetilde{\bfL}_3 + E_{\bfP_{\!3}} - \bfQ_3 \big|.
	$$
	
	Next, let $E_{\bfQ_3}$ be the exceptional divisor associated with the point set $\bfQ_3$ and $\str{\bfL}_3$ (resp. $\wt{E}_{\bfP_{\!3}}$) be the proper preimage of $\wt{\bfL}_3$ (resp. $E_{\bfP_{\!3}}$) with respect to $bl_{\bfQ_{3}}$. We have:
	$$
	bl_{\bfQ_3}^*\big( \wt{\L} \big) = \big| bl_{\bfQ_3}^*\big( \wt{\bfL}_3 + E_{\bfP_{\!3}} \big) - E_{\bfQ_3} \big| + E_{\bfQ_3} = \ \str{\!\L} + E_{\bfQ_3},
	$$
	$$
	\L^{**} := bl_{\bfQ_3}^*(\L^*) = bl_{\bfQ_3}^*\big( \wt{\L} \big) + bl_{\bfQ_3}^*(E_{\bfP_{\!3}}) = \ \str{\!\L} + \wt{E}_{\bfP_{\!3}} + 2E_{\bfQ_3},
	$$
	where $\ \str{\!\L} \ := \big|\! \str{\bfL}_3 + \wt{E}_{\bfP_{\!3}} + E_{\bfQ_3} \big|$.
	
	From the identities
	$$
	(\varphi_{min})_*\Big(\!\! \stackrel{\approx}{\bfL}_3 \!\!\Big) = (\varphi_{min})_*\big( \wt{E}_{\bfP_{\!3}} \big) = 0,\qquad (\varphi_{min})_*(E_{\bfQ_3}) = \h{\bfL}_3 \sim {-}K_{S_3}
	$$
	it follows that
	$$
	\varphi^*(\L) = (\varphi_{min})_* (\L^{**}) = 
	(\varphi_{min})_* \Big(\! \str{\!\L} \!\Big) + (\varphi_{min})_*\big( \wt{E}_{\bfP_{\!3}} + 2E_{\bfQ_3} \big) = |{-}K_{S_3}| + 2\h{\bfL}_3.
	$$
	Finally, $\L$ has the geometric description declared in the corollary by virtue of \cite[Exercise V.3.2]{Hartshorne}.
\end{proof}

\begin{figure}
	\begin{minipage}{0.49\linewidth}
		\centering
		\begin{tikzpicture}[scale=1.6]
		\begin{scope}[very thick]
		\draw (-1, -1.5) node[right] {$L_0$} -- (0.5, 1.25);
		\draw (1, -1.5) node[right] {$L_2$} -- (-0.5, 1.25);
		\draw (-1.5, -0.75) node[above] {$L_1$} -- (1.5, -0.75);
		
		\draw (0.23, -0.57) arc (100: -40: 0.35);
		\draw (0.04, -0.09) arc (220: 80: 0.35);
		\draw (-0.92, -0.98) arc (160: 20: 0.35) node[below]{$C$};
		\end{scope}
		
		\fill (0, 0.35) node[left] {$P_1\,$};
		\fill (-0.6, -0.75) node[left] {\raisebox{0.7cm}{$P_2\!$}};
		\fill (0.6, -0.75) node[right] {\raisebox{0.7cm}{$\!P_0$}};
		
		\begin{scope}[red, densely dashed]	
		\draw[-<-=.35] (-1.25, -0.75) arc (180: 255: 0.5);
		\draw[-<-=.35] (0.9, -1.25) arc (285: 360: 0.5);
		\draw[-<-=.35] (0.35, 1) arc (45: 135: 0.5);
		\end{scope}
		\end{tikzpicture}
		\caption{The lines $L_i \subset \P^2$ and a cubic \mbox{$C \in \L$}, $C \neq \bfL_3$}
		\label{LinesLionP2AndCurveCFromL}
	\end{minipage}
	\begin{minipage}{0.49\linewidth}
		\centering
		\begin{tikzpicture}[scale=1.5]
		\begin{scope}[very thick]	
		\draw (-1.25, 1) -- (1.25, 1) node[above]{$E_{P_1}$};
		\draw (-1.25, -1) -- (1.25, -1) node[below]{$\widetilde{L}_1$};
		\draw (-0.5, -1.5) -- (-1.5, 0.5) node[left]{$E_{P_2}$};
		\draw (-1.5, -0.5) node[left]{$\widetilde{L}_0$} -- (-0.5, 1.5);
		\draw (0.5, -1.5) node[right]{$E_{P_0}$} -- (1.5, 0.5);
		\draw (1.5, -0.5) -- (0.5, 1.5) node[right]{$\widetilde{L}_2$};
		
		\draw (-0.8, -1.45) node[left] {$\wt{C}$} arc (200: 90: 0.5);
		\draw (-0.8, 1.45) arc (160: 270: 0.5);
		\draw (0.85, 0.2) arc (215: 325: 0.5);
		\end{scope}	
		
		\fill (-0.7, -1) node[left] {\raisebox{0.7cm}{$Q_2\ \,$}};
		\fill (-0.7, 1) node[left] {\raisebox{-0.9cm}{$Q_1 \hspace{0.25cm}$}};
		\fill (1.24, 0) node[below] {\raisebox{-0.7cm}{$Q_0$}};
		
		\fill (0.7, -1) node[left] {\raisebox{0.7cm}{$Q_2^\prime\!\!$}};
		\fill (0.7, 1) node[left] {\raisebox{-1cm}{$Q_0^\prime\!\!\!$}};
		\fill (-1.3, 0) node[left] {$Q_1^\prime$};
		
		\begin{scope}[red, densely dashed]
		\draw[-<-=.25] (0, 1) -- (1, -0.5);
		\draw[-<-=.25] (1, -0.5) -- (-1, -0.5);	
		\draw[-<-=.25] (-1, -0.5) -- (0, 1);
		\end{scope}
		\begin{scope}[blue, densely dashed]
		\draw[-<-=.5] (-1, 0.5) .. controls (0, 0.25) .. (1, 0.5);
		\draw[-<-=.5] (1, 0.5) .. controls (0.25, 0) and (0.25, -0.25) .. (0, -1);
		\draw[-<-=.5] (0, -1) .. controls (-0.25, -0.25) and (-0.25, 0) .. (-1, 0.5);
		\end{scope}
		\end{tikzpicture}
		
		\caption{The curves $E_{P_i}, \widetilde{L}_i, \widetilde{C} \subset \D_3$}
		\label{CurvesOnSurfaceD3}
	\end{minipage}	
\end{figure}

One can easily check that

\begin{remark}
	Any $C \in \L$ different from $\bfL_3$ is an absolutely irreducible (possibly singular) $\F{q}$-cubic.
\end{remark}

\begin{remark} \label{remarkEllipticCurvesS3}
	Given $C \in \L$, the divisor $D := \varphi^*(C) - 2\h{\bfL}_3$ is an elliptic curve if and only if $C$ is one too. Moreover, in this case $\varphi\!: D \to C$ is an isomorphism.
\end{remark}

\begin{lemma}
	Let $E \in \L$ be an elliptic $\F{q}$-curve and $\O \in E(\F{q^3})$ be one of its flexes, which, as is known, always exists over $\F{q^3}$ $($for details see \mbox{\rm\cite[Chapter 11]{Hirschfeld1998})}. Then the $P_i$ are points of order $9$ $($with respect to $\O$ as the neutral element of the chord-tangent group law on $E)$ such that $\langle P_0 \rangle = \langle P_1 \rangle = \langle P_2 \rangle.$
\end{lemma}

\begin{proof}
	By definition of $\L$ in Corollary \ref{AnticanLinearSystemS3}, 
	$$
	2P_0 + P_1 \: = \: 2P_1 + P_2 \: = \: 2P_2 + P_0 \: = \: \O,
	$$
	hence we see that
	$$
	9P_0 = \O,\qquad 7P_0 = P_1,\qquad 4P_0 = P_2.
	$$
	Similarly, $P_0$, $P_2$ (resp. $P_0$, $P_1$) are expressed through $P_1$ (resp. $P_2$). Finally, the points $P_i$ are of order $9$, otherwise they would be equal. 
\end{proof}

\begin{theorem}
	For any elliptic $\F{q}$-curve $E \in |{-}K_{S_3}|$ the order $|E(\F{q})|$ is divisible by $3$.
\end{theorem}	

\begin{proof}
	The result is proved by exhibiting an $\F{q}$-point of order $3$ on $E$. 
	
	By Remark \ref{remarkEllipticCurvesS3} we are in the conditions of the previous lemma, i.e., up to an $\F{q}$-isomorphism $E \in \L$. This curve has the group structure with respect to any point $\O^\prime \in E(\F{q}) \neq \emptyset$ \big(instead of a flex $\O \in E(\F{q^3})$\big) as the neutral element. It is well known that there is the group $\F{q^3}$-isomorphism $\tau(P) := P + \O^\prime$, \mbox{$\tau\!: E \simTo E.$} At the same time, $E$ has a Weierstrass form $W\!: y^2 + h(x)y = f(x)$ defined over $\F{q}$, where $\deg(h) \leqslant 1$, $\deg(f) = 3$. Let $\sigma\!: E \simTo W$ be the corresponding $\F{q}$-isomorphism such that $\sigma(\O^\prime)$ is the point at infinity.
	
	Consider the $\F{q^3}$-point $(x_0,y_0) := (\sigma \circ \tau)(3P_0)$ of order $3$ on $W$. If $x_0 \in \F{q}$ (e.g., this is true for $p=3$), then $y_0 \in \F{q^3} \cap \F{q^2} = \F{q}$ and all is proved. Otherwise the $3$-division polynomial $\psi_3$  (see, e.g., \cite[Exercise 3.7]{Silverman}) has exactly two $\F{q}$-irreducible factors, namely the $\F{q}$-minimal (cubic) polynomial of $x_0$ and $x - x_1$ for some $x_1 \in \F{q}$. Note that the $3$-torsion subgroup $W[3]$ is generated, for example, by the points $(x_0,y_0)$, $(x_0^q,y_0^q)$. Therefore $W[3] \subset W(\F{q^3})$ and thus $(x_1,y_1) \in W(\F{q})[3]$ for an appropriate $y_1$.
\end{proof}

\begin{corollary} \label{supersingularS3}
	For $p=3$ supersingular elliptic curves $($i.e., of $j$-invariant $0)$ \mbox{\rm\cite[\S 2.4.3]{TsfasmanVladutNogin}} do not belong to $|{-}K_{S_3}|$.
\end{corollary}

Finally, carefully analyzing small values $q$, we get the following result.

\begin{corollary} \label{FqPointsAntiCanDivisorsS3}
	For $q \geqslant 3$ and any $\F{q}$-divisor $D \in |{-}K_{S_3}|$ we have 
	$$
	|\Supp(D)(\F{q})| \leqslant 3\lfloor N_q(1)/3 \rfloor,
	$$ 
	where the number $N_q(1)$ is given in Theorem {\rm\ref{NqOne}}.
\end{corollary}

\subsection{Toric (singular) intersection of two quadrics in $\P^4$} \label{surfaceS4}

Let us fist suppose that $p > 2$. Choose quadratic non-residues $b \in \F{q}$ and $a := a_0 + a_1\sqrt{b} \in \F{q^2}$ (for some $a_0, a_1 \in \F{q}$) and consider the following intersection of two $\F{q}$-quadrics:
$$
S_4\!: \begin{cases}
x_0^2 + bx_1^2 - a_0(y_0^2 + by_1^2) - 2a_1by_0y_1 = z^2,\\
\EQS\!: 2x_0x_1 - a_1(y_0^2 + by_1^2) - 2a_0y_0y_1 = 0
\end{cases}
\subset \quad \P^4_{(x_0:x_1:y_0:y_1:z)}.
$$

\noindent Note that the affine open subset $U := S_4 \setminus \{z = 0\}$ is the Weil restriction (with respect to the extension $\F{q^2}/\F{q}$) of the $\F{q^2}$-conic 
$$
C_2\!: x^2 - a y^2 = 1 \ \subset \ \A{2}_{(x,y)} \qquad \textrm{if} \qquad 
x = x_0 + x_1\sqrt{b},\qquad y = y_0 + y_1\sqrt{b}.
$$
At the same time, $C_2$ is isomorphic to the torus $T_2$ (\cite[Theorem 7]{MyArticle2019}).

For $p=2$ we can take elements $b \in \F{q}$, $a \in \F{q^2}$ such that $\Tr_{\mathbb{F}_q/\mathbb{F}_2}(b) = \Tr_{\F{q^2}/\mathbb{F}_2}(a) = 1$. As is well known, the equation $x^2 + x + b$ (resp. $x^2 + x + a$) has no roots over $\F{q}$ (resp. $\F{q^2}$). Thus there are not any problems to write out the equations of $C_2$ and $S_4$ in even characteristic.

For $i \in \Z/4$ we enumerate the lines $\h{L}_i$ of $S_4 \cap \{z=0\}$ such that $\h{P}_i := \h{L}_{i+1} \cap \h{L}_{i+2}$ is a point. Also, let
$$
L_i := pr(\h{L}_i),\qquad P_i := pr(\h{P}_i) = L_{i+1} \cap L_{i+2}, \qquad
\textrm{where} \qquad pr\!: S_4 \to \EQS \ \subset \ \P^3_{(x_0:x_1:y_0:y_1)} 
$$
is the well-defined projection of degree $2$ onto the {\it elliptic quadratic surface} $\EQS$ (from Table \ref{toricSurfaces}). Finally, let 
$$
\h{\bfL}_4 := \sum_{i=0}^3\h{L}_i \qquad \mathrm{and} \qquad \bfP_{\!4} := \{P_i\}_{i=0}^3.
$$

\begin{remark}
	The surface $S_4$ is toric with respect to the torus $T_4 \simeq U$ $($see {\rm\cite[Theorem 8]{MyArticle2019})} and the lines $\widehat{L}_i$ $($resp. $\widehat{P}_i)$ are the unique $T_4$-invariant curves $($resp. points$)$ on $S_4$. Moreover, they are $\F{q}$-conjugate.
\end{remark}

\noindent Recall that the surface $\EQS$ is also toric for $T_4$. 

\begin{lemma}[\mbox{\cite[Table 6]{Derenthal}}]
	We have:
	\begin{enumerate} 
		\item The points $\widehat{P}_i$ are the unique singularities on $S_4$ $($of type $\rmA_1);$
		\item $\varphi_{min}\!: S_4^\prime \to S_4$ is the simultaneous blowing up at them;
		\item $\widehat{L}_i$ are the unique lines on $S_4$. 
	\end{enumerate}
\end{lemma}

The following theorem is an analogue of Theorem \ref{theSurfaceS3}, hence its statements are proved in a similar way. Unfortunately, we did not find quite exact references.

\begin{theorem} \label{theSurfaceS4} We have:
	\begin{enumerate}
		\item $S_4$ is the unique $($up to an $\F{q}$-isomorphism$)$ toric del Pezzo surface of degree $4$ with respect to the torus $T_4;$
		\item $S_4$ is the non-split toric surface associated with the pair $(Pol_4, \varPhi_4^{\,t});$
		\item \label{statement3S4} $S_4$ is the quotient $\EQS/\sigma$ under an automorphism $\sigma$ of $\EQS$ $\big($in particular, $\sigma \in \mathrm{PGL}(4, \F{q})\big)$, whose fixed point set is $\bfP_{\!4};$ 
		\item \label{statement4S4} $S_4^\prime$ is the blowing up of $\EQS$ at the set $\bfP_{\!4}$ of all $($i.e., four $\F{q}$-conjugate$)$ $T_4$-invariant points.
	\end{enumerate}	
\end{theorem}

From \cite[Theorem 14]{MyArticle2019} or one of Statements \ref{statement3S4}, \ref{statement4S4} of Theorem \ref{theSurfaceS4} it follows that the Picard $\F{q}$-number of $S_4$ is equal to $1$. Since the Fano index of $S_4$ is also $1$, we obtain

\begin{lemma} The Picard $\F{q}$-group of the surface $S_4$ is equal to
	$$
	\Pic(S_4) = \langle{-}K_{S_4}\rangle \simeq \Z.
	$$
\end{lemma}

It is well known that besides $T_4$ the surface $\EQS$ is toric for the torus $T_{2.c}$ (from \cite[Theorem 8]{MyArticle2019}). Let $M_i$ be the lines outside $T_{2.c}$ and $R_i$ be their intersection points. The blowing up of $\EQS$ at the $\F{q}$-point $R_0$ (or $R_2$) gives the nonsingular del Pezzo surface $\D_7$ of degree $7$, which is also the blowing up of $\P^2$ at a pair $\bfQ_2 = \{Q_1, Q_2\}$ of $\F{q}$-conjugate points. Thus we have the diagram
$$
S_4 \stackrel{ \raisebox{0.2cm}{$\varphi_{min}$} }{\longleftarrow} S_4^\prime \stackrel{ \raisebox{0.2cm}{$bl_{\bfP_{\!4}}$} }{\longrightarrow} \EQS \stackrel{ \raisebox{0.2cm}{$bl_{R_0}$} }{\longleftarrow} \D_7 \stackrel{ \raisebox{0.2cm}{$bl_{\bfQ_{2}}$} }{\longrightarrow} \P^2,
$$
where $bl_{\bfP_{\!4}}$, $bl_{R_0}$, $bl_{\bfQ_2}$ are the corresponding blowing up maps. Besides, let 
$$
\chi := bl_{\bfP_{\!4}} \circ \inv{\varphi_{min}} \qquad \chi\!: S_4 \dashrightarrow \EQS 
,\qquad\qquad
\psi := bl_{\bfQ_{2}} \circ \inv{bl_{R_0}}\qquad \psi\!: \EQS \dashrightarrow \P^2
,\qquad\qquad 
$$
$$
\varphi := \psi \circ \chi \qquad \varphi\!: S_4 \dashrightarrow \P^2.
$$
Further,
$$
\h{M}_i := \chi^*(M_i), \qquad \h{\bf{M}}_2 := \h{M}_1 + \h{M}_2, \qquad\qquad \widetilde{L}_i := \psi_*(L_i), \qquad \widetilde{\bfL}_4 := \sum_{i=0}^3 \widetilde{L}_i.
$$
Finally, $L$ is the line through the points $Q_j = \psi(M_j)$ and also we identify $P_i$ with $\psi(P_i)$. For more clarity on what is going on, see Figures \ref{linesLionP2}, \ref{linesLiMionEQS}, where arrows denote the Frobenius action. 

Repeating the arguments used for the proof of Corollary \ref{AnticanLinearSystemS3}, we obtain the following result. Let us not write out its proof, because it is also very technical and does not contain new ideas.

\begin{corollary} The anticanonical linear system of $S_4$ is equal to
	$$
	|{-}K_{S_4}| = \varphi^*(\L) - \h{\bfL}_4 - 2\h{\bf{M}}_2, \qquad \textrm{where} \qquad \L := |\widetilde{\bfL}_4 - \bfP_{\!4} - 2\bfQ_2|
	$$
	is the (incomplete) linear system of all (possibly reducible or singular) quartics $C \subset \P^2$ passing through $\bfP_{\!4}$ and through $\bfQ_2$ with multiplicity at least $2$.
\end{corollary}

\begin{figure}
	\begin{minipage}{0.49\linewidth}
		\centering
		\begin{tikzpicture}[scale=1.25]
		\begin{scope}[very thick]	
		\draw (1.7, 2) node[right] {\hspace{0.1cm}$\widetilde{L}_{2}$} -- (3.75, -1.75);
		\draw (-0.5, -0.4) node[below] {$\widetilde{L}_{1}$} -- (4, 1.5);
		\draw (-0.5, 0.4) node[above] {$\widetilde{L}_{0}$} -- (4, -1.5); 	
		\draw (1.7, -2) node[right] {\hspace{0.1cm}$\widetilde{L}_{3}$} -- (3.75, 1.75); 								
		\draw (3.5, 2) node[left] {$L$} -- (3.5, -2);
		\end{scope}
		
		\fill (2.35, 0.8) node[left] { \raisebox{0.4cm}{\hspace{-1cm} $P_0$} } circle (0.075);
		\fill (2.8, 0) node[right] {$P_1$} circle (0.075);
		\fill (2.35, -0.8) node[left] { \raisebox{-0.6cm}{\hspace{-1cm} $P_2$} } circle (0.075);	
		\fill (0.45, 0) node[above] { \raisebox{0.1cm}{$P_3$} } circle (0.075);
		\fill (3.5, 1.3) node[right] { \raisebox{-0.8cm}{$Q_2$} } circle (0.075);
		\fill (3.5, -1.3) node[right] { \raisebox{0.6cm}{$Q_1$} } circle (0.075);	
		
		\begin{scope}[red, densely dashed]
		\draw[->-=.7] (1.6, 0.5) .. controls (1.9, 0.1) and (2.3, 0.1) .. (2.6, 0.4);
		\draw[->-=.8] (3.1, -0.5) .. controls (3.5, -0.25) and (3.5, 0.25) .. (3.1, 0.5);
		\draw[->-=.7] (1.25, -0.3) .. controls (1.3, -0.2) and (1.3, 0.2) .. (1.25, 0.3);
		\draw[->-=.7] (2.6, -0.4) .. controls (2.3, -0.1) and (1.9, -0.1) .. (1.6, -0.5);
		\draw[->-=.8] (3.5, 0) .. controls (4.25, -0.5) and (4.25, 0.5) .. (3.5, 0); 
		\end{scope}
		\end{tikzpicture}
		\caption{The lines $\widetilde{L}_i, L \subset \P^2$}
		\label{linesLionP2}
	\end{minipage}	
	\begin{minipage}{0.49\linewidth}
		\centering
		\begin{tikzpicture}[scale=1.25]
		\begin{scope}[very thick]
		\draw (-1.5, -2) node[right] {$L_0$} -- (-1.5, 2);
		\draw (-0.5, -2) node[right] {$L_2$} -- (-0.5, 2);
		\draw (0.5, -2) -- (0.5, 2) node[left] {$M_0$};
		\draw (1.5, -2) -- (1.5, 2) node[left] {$M_2$};
		
		\draw (-2, -1.5) node[above] {$M_3$} -- (2, -1.5);
		\draw (-2, -0.5) node[above] {$M_1$} -- (2, -0.5);
		\draw (-2, 0.5) -- (2, 0.5) node[below] {$L_3$};
		\draw (-2, 1.5) -- (2, 1.5) node[below] {$L_1$};	
		\end{scope}
		
		\fill (-1.5, 1.5) node[left] { \raisebox{0.7cm}{$P_3$} } circle (0.075);
		\fill (-0.5, 1.5) node[left] { \raisebox{0.7cm}{$P_0$} } circle (0.075);
		\fill (-0.5, 0.5) node[right] { \raisebox{-1cm}{$P_1$} } circle (0.075);
		\fill (-1.5, 0.5) node[left] { \raisebox{0.7cm}{$P_2$} } circle (0.075);	
		
		\fill (1.5, -1.5) node[right] { \raisebox{-1cm}{$R_1$} } circle (0.075);
		\fill (0.5, -1.5) node[right] { \raisebox{-1cm}{$R_2$} } circle (0.075);
		\fill (0.5, -0.5) node[left] { \raisebox{0.7cm}{$R_3$} } circle (0.075);
		\fill (1.5, -0.5) node[right] { \raisebox{-1cm}{$R_0$} } circle (0.075);
		
		\begin{scope}[red, densely dashed]
		\draw[->-=.7] (-1, 1.5) -- (-0.5, 1);
		\draw[->-=.7] (-0.5, 1) -- (-1, 0.5);
		\draw[->-=.7] (-1, 0.5) -- (-1.5, 1);
		\draw[->-=.7] (-1.5, 1) -- (-1, 1.5);
		
		\draw[-<-=.35, ->-=.75] (1, -1.5) arc (0: 90: 0.5);
		\draw[-<-=.35, ->-=.75] (1, -0.5) arc (180: 270: 0.5);
		\end{scope}
		\end{tikzpicture}
		\caption{The lines $L_i, M_i \subset \EQS$}
		\label{linesLiMionEQS}
	\end{minipage}
\end{figure}

One can easily check that

\begin{remark}
	Any $C \in \L$ contains at most one absolutely irreducible $\F{q}$-curve $($of geometric genus $g \leqslant 1)$ different from $L$.
\end{remark}

\begin{remark} \label{ellipticCurvesFromL}
	Given $C \in \L$, the divisor $D := \varphi^*(C) - \h{\bfL}_4 - 2\h{\bf{M}}_2$ is an elliptic curve if and only if $C$ is one of the following quartics:
	\begin{enumerate}
		\item $E \cup L$, where $E \subset \P^2$ is an elliptic curve passing through $\bfP_{\!4}, \bfQ_2;$
		\item \label{ellipticQuarticFromL} An irreducible quartic for which $Q_1, Q_2$ are the unique singularities (namely nodes).
	\end{enumerate}
	Moreover, in the first case $\varphi\!: D \to E$ is an isomorphism and in the second one $\varphi\!: D \to C$ is the blowing up at $\bfQ_2$ such that $|\inv{\varphi}(Q_j)| = 2$.	
\end{remark}

\begin{lemma} \label{cubicCurveThroughP4Q2}
	Let $E \subset \P^2$ be an elliptic $\F{q}$-curve passing through $\bfP_{\!4}, \bfQ_2$ and $\O \in E(\F{q})$. Then
	$$P_0-P_2 = P_1-P_3 = Q_1-Q_2$$ 
	is an $\F{q}$-point of order $2$ $($with respect to $\O$ as the neutral element of the chord-tangent group law on $E)$.
\end{lemma}

\begin{proof}
	By definition, 
	$$
	P_0 + P_1 + Q_1 \: = \: P_1 + P_2 + Q_2 \: = \: P_2 + P_3 + Q_1 \: = \:
	P_3 + P_0 + Q_2.
	$$
	Therefore
	$$
	P_0 + P_1 = P_2 + P_3,\qquad P_1 + P_2 = P_3 + P_0,\qquad P_0 + Q_2 = P_2 + Q_1
	$$
	and hence 
	$$
	P_0 - P_2 \: = \: P_3 - P_1 \: = \: P_2 - P_0 \: = \: Q_2 - Q_1.
	$$
	This is an $\F{q}$-point, because $Q_1,Q_2$ are $\F{q}$-conjugate.
\end{proof}

A quartic $C \in \L$ from Remark \ref{ellipticCurvesFromL}.\ref{ellipticQuarticFromL} gives a geometric interpretation of the group law on the elliptic curve $D$. Note that $C$ is similar to a ({\it twisted}) {\it Edwards quartic} \cite{Bernstein}, because both curves have two nodes. The group law on the latter is represented in \cite[\S 4]{Arene}. An analog for $C$ is defined in the following way. 

For points $R_1, R_2 \in C \setminus \bfQ_2$ let $R_1\cd R_2$ be the eighth intersection point of $C$ with the unique conic passing through $R_1, R_2, P_1, Q_1, Q_2$. If $R_1 = R_2$ (resp. $R_1 = P_1$ or $R_2 = P_1$), then this conic intersects $C$ at $R_1$ (resp. $P_1$) with the intersection number at least $2$ ($3$ if $R_1=R_2=P_1$). Besides, let $\overline{R_1}$ be the third intersection point of $C$ with the unique line passing through $R_1, Q_2$. The points $R_1\cd R_2$ and $\overline{R_1}$ are correctly defined by \cite[\S 3.3, \S 5.3]{Fulton}. Then the addition and subtraction (with $P_0$ as the neutral point) have the form
$$
R_1 + R_2 := \overline{R_1\cd R_2}, \qquad -R_1 := R_1\cd P_3
$$
respectively. This can be proved in the same way as \cite[Theorem 2]{Arene}. Note that sometimes $R_1+R_2$ or $-R_1$ falls into $\bfQ_2$. Finally, $P_2$ is obviously a point of order~$2$.

\begin{theorem}
	For any elliptic $\F{q}$-curve $E \in |{-}K_{S_4}|$ the order $|E(\F{q})|$ is even.
\end{theorem}

\begin{proof}
	The result is proved by exhibiting an $\F{q}$-point of order $2$ on $E$. 
	
	By Lemma \ref{cubicCurveThroughP4Q2} it remains to only consider the case with a quartic $C \in \L$ from Remark \ref{ellipticCurvesFromL}.\ref{ellipticQuarticFromL}. The corresponding elliptic curve $D$ has the group structure with respect to any point $\O \in D(\F{q}) \neq \emptyset$ \big(instead of $\inv{\varphi}(P_0)$\big) as the neutral element. It is well known that there is the group $\F{q^4}$-isomorphism $\tau(P) := P + \O$, $\tau\!: D \simTo D.$ At the same time, $D$ has a Weierstrass form $W\!: y^2 + h(x)y = f(x)$ defined over $\F{q}$, where $\deg(h) \leqslant 1$, $\deg(f) = 3$. Let $\sigma\!: D \simTo W$ be the corresponding $\F{q}$-isomorphism such that $\sigma(\O)$ is the point at infinity.
	
	Consider the $\F{q^4}$-point $(x_0,y_0) := (\sigma \circ \tau \circ \inv{\varphi})(P_2)$ of order $2$ on $W$. For $p=2$ it is the only such point, hence it is defined over $\F{q}$. For $p>3$, as is known, $y_0 = f(x_0) = 0$. If $x_0 \in \F{q}$, then all is proved. Otherwise $f(x)$ has exactly two $\F{q}$-irreducible factors, namely the $\F{q}$-minimal (quadratic) polynomial of $x_0$ and $x - x_1$ for some $x_1 \in \F{q}$. Thus $(x_1,0) \in W(\F{q})[2]$.
\end{proof}

\begin{corollary} \label{supersingularS4}
	For $p=2$ supersingular elliptic curves $($i.e., of $j$-invariant $0)$ \mbox{\rm\cite[\S 2.4.3]{TsfasmanVladutNogin}} do not belong to $|{-}K_{S_4}|$.
\end{corollary}

Finally, carefully analyzing small values $q$, we get the following result.

\begin{corollary} \label{FqPointsAntiCanDivisorsS4}
	For any $\F{q}$-divisor $D \in |{-}K_{S_4}|$ we have 
	$$
	|\Supp(D)(\F{q})| \leqslant 2\lfloor N_q(1)/2 \rfloor,
	$$ 
	where the number $N_q(1)$ is given in Theorem {\rm\ref{NqOne}}.
\end{corollary}

\section{BCH codes}

Let us recall some notions of {\it BCH codes} over an arbitrary finite field $\F{q}$. Let $n, d^*, b \in \N$, where $d^*$ is so-called {\it designed distance}. Also, let $\alpha$ be a primitive $n$-th root of unity and $e := [\F{q}(\alpha) : \F{q}]$. $\BCH_q(n, d^*, b)$ is a cyclic code given by the generator polynomial 
$$g(x) = \mathrm{LCM}(m_{\alpha^b}, m_{\alpha^{b + 1}}, \cdots\!, m_{\alpha^{b + d^* - 2}}),$$
where $m_{\alpha^i}$ is the $\F{q}$-minimal polynomial of $\alpha^i$. A BCH code is said to be {\it primitive} (resp. {\it narrow-sense}) if $n = q^e - 1$ (resp. $b = 1$). The theory of BCH codes is well represented, for example, in \cite[\S 9]{MacWilliamsSloane}.

\begin{theorem}[{\cite[Theorem 9.1.a]{Stichtenoth}}] For a $\BCH_q(n, d^*, b)$ code we have
	$$
	k \geqslant n - e(d^*-1),\qquad d \geqslant d^*.
	$$
\end{theorem}

\noindent The second inequality is called the {\it BCH bound}.

\begin{theorem}[{\cite[Proposition 2.3.9]{Stichtenoth}}] 
	Let 
	$$r := b-1,\qquad s := n + 1 - d^* - b,\qquad P_0 := (0:1),\qquad P_\infty := (1:0).$$
	A $\BCH_q(n, d^*, b)$ code is obtained by the successive puncturing of the split toric code $\C_{q^e}(\P^1, \G_m, rP_0 + sP_\infty)$ $($see the notation in {\rm\cite[\S 3.1]{MyArticle2019})} at the coordinate set $\sqrt[n]{1}$ and the restriction to $\F{q}$ {\rm(}or in another order{\rm)}. 
\end{theorem}	

\begin{corollary}
	A primitive narrow-sense $\BCH_q(q^{e}-1, d^*, 1)$ code is the restriction to $\F{q}$ of the Reed--Solomon $\F{q^e}$-code of length $q^e-1$ and dimension $q^e-d^*$.
\end{corollary}

\section{Codes associated with the \grqq symmetric'' polygons}

Next we will need the following facts. 

\begin{theorem}[Griesmer bound \mbox{\cite[Theorem 1.1.43]{TsfasmanVladutNogin}}]
	For any linear $[n,k,d]_q$ code we have
	$$
	\delta := n - \sum_{i=0}^{k-1}\left\lceil \dfrac{d}{q^i} \right\rceil \geqslant 0.
	$$	
\end{theorem}

\begin{theorem}[\mbox{\cite[Theorem 3.4.49]{TsfasmanVladutNogin}}] \label{NqOne}
	The maximal possible number of $\F{q}$-points on an elliptic $\F{q}$-curve is equal to
	$$
	N_q(1) = 
	\begin{cases}
	\hspace{-0.2cm}
	\begin{array}{ll}
	q + \lfloor 2\sqrt{q} \rfloor & \textrm{if }\: \sqrt{q} \notin \N,\ p < q, \textrm{ and }\ p \mid \lfloor 2\sqrt{q} \rfloor, \\
	q + \lfloor 2\sqrt{q} \rfloor + 1 & \textrm{otherwise}.
	\end{array} \qquad  
	\end{cases}
	$$
\end{theorem}

\noindent For small $q$ Table \ref{optimalEllipticCurves} (the original source is \cite{vanderGeer}) contains Weierstrass forms of $\F{q}$-optimal elliptic curves, i.e., having $N_q(1)$ points over $\F{q}$. According to \cite[Theorem 4.6, Table I]{Schoof} these curves are unique (up to $\F{q}$-isomorphism) among $\F{q}$-optimal. The last column of the table is filled by \cite[Proposition 3.6.iv]{Schoof}. 

\begin{table}[h]	
	\centering
	\begin{tabular}{c|c|c|c|c}
		
		$q$ & $N_q(1)$ & elliptic $\F{q}$-curve & $j$-invariant & is supersingular \\
		\hline
		\hline
		
		$2$ & $5$ & $y^2+y=x^3+x$ & $0$ & yes \\ \hline
		$3$ & $7$ & $y^2=x^3+2x+1$ & $0$ & yes \\ \hline
		$4$ & $9$ & $y^2+y=x^3$ & $0$ & yes \\ \hline
		$5$ & $10$ & $y^2=x^3+3x$ & $1728$ & no \\ \hline
		$7$ & $13$ & $y^2=x^3+3$ & $0$ & no \\ \hline
		$8$ & $14$ & $y^2+xy+y=x^3+1$ & $1$ & no \\ \hline
		$9$ & $16$ & $y^2=x^3+x$ & $0$ & yes 
	\end{tabular}
	\caption{$\F{q}$-optimal elliptic curves for small $q$}
	\label{optimalEllipticCurves}	
\end{table}

For $i \in \{3, 4, 6, 8, 9\}$ by $\C_i$ we will denote the non-split toric $\F{q}$-code associated with the polygon $Pol_i$ from Figure \ref{symmetricReflexivePolygons}. In other words, $\C_i$ are anticanonical codes on the non-split toric del Pezzo $\F{q}$-surfaces from Table \ref{toricSurfaces}. In particular, $\C_9$ is equivalent to the so-called {\it projective Reed--Muller code}. The code parameters are represented in Table \ref{ToricCodesOnDelPezzoSurfaces} (for a value $q$ satisfying the restriction). The bound on $d$ for the new codes $\C_3$, $\C_4$ follows from Corollaries \ref{FqPointsAntiCanDivisorsS3}, \ref{FqPointsAntiCanDivisorsS4}.  Be careful that for very small $q$ (even if the restriction is satisfied) values of the column $\delta$ may be incorrect. 

\begin{table}[h]
	\centering
	{\footnotesize
		\begin{tabular}{c|c|c|c|c|c|c}
			code & $n$ & $k$ & $d$ & restriction & $\delta$ & reference \\
			\hline
			\hline
			
			$\C_3$ & $q^2 + q + 1$ & $4$ & $\geqslant n - 3\lfloor N_q(1)/3 \rfloor$ & $3 \leqslant q$ & $\leqslant 3\lfloor N_q(1)/3 \rfloor - q - 2$ & \multirow{2}{*}{new codes} \\
			\cline{1-6}
			
			$\C_4$ & $q^2 + 1$ & $5$ & $\geqslant n - 2\lfloor N_q(1)/2 \rfloor$ & & $\leqslant 2\lfloor N_q(1)/2 \rfloor - q - 2$ & \\
			\hline
			
			$\C_6$ & $q^2 - q + 1$ & $7$ & $n - N_q(1)$ & $5 \leqslant q$ & $N_q(1) - q - 3$ & \cite[Cor.\! 4]{MyArticle2019} \\
			\hline
			
			$\C_8$ & $q^2 + 1$ & $9$ & $n - 2(q+1)$ & $3 \leqslant q$ & $q-3$ & \cite[Prop.\! 4.7]{CouvreurDuursma} \\
			\hline
			
			$\C_9$ & $q^2 + q + 1$ & $10$ & $n - (3q + 1)$ & $5 \leqslant q$ & $2q - 5$ & \cite[\S 2]{Lachaud1990} \\
			
	\end{tabular}}
	\caption{The non-split toric codes on the polygons of Figure \ref{symmetricReflexivePolygons}}
	\label{ToricCodesOnDelPezzoSurfaces}
\end{table}

%

For $n, i \in \N$ let $\alpha \in \overline{\F{q}}$ be an element of order $n$ and $m_{\alpha^i}$ be the $\F{q}$-minimal polynomial of $\alpha^i$. In Table \ref{ParityCheckPolynomials} by means of \cite[Theorem 25]{MyArticle2019} it is written the parity-check polynomials $h(x)$ of the codes from Table \ref{ToricCodesOnDelPezzoSurfaces}. It is immediately checked that these codes are $\BCH_q(n, d^*, b)$ ones. Finally, the column LCD answers whether a cyclic code is a {\it linear code with complementary dual} (or, equivalently, {\it reversible}) or not (details see in \cite{YangMassey}). It is filled, looking at $h(x)$, but \grqq yes'' also follows from \cite[Corollary 3]{MyArticle2019} or \cite[Problem 7.27]{MacWilliamsSloane}. As a result, the dual codes to $\C_4$, $\C_6$ are $\BCH_q(n, 4, n-1)$ codes.

\begin{table}[h]
	\centering
	{\small
		\begin{tabular}{c|c|c|c|c|c}
			code & $h(x)$ & $d^*$ & $b$ & $d - d^*$ & LCD 
			\\ \hline\hline
			
			$\C_3$ & \multirow{3}{*}{$(x-1)\cdot m_{\alpha}$} & $q^2 - q$ & \multirow{3}{*}{$q+1$} & $\geqslant 2q + 1 - 3\lfloor N_q(1)/3 \rfloor$ & no 
			\\ \cline{1-1}\cline{3-3}\cline{5-6}
			
			$\C_4$ & & $q^2 - 2q + 1$ & & $\geqslant 2q - 2\lfloor N_q(1)/2 \rfloor$ & yes 
			\\ \cline{1-1}\cline{3-3}\cline{5-6}
			
			$\C_6$ & & $q^2 - 3q + 1$ & & $2q - N_q(1)$ & yes 
			\\ \hline
			
			$\C_8$ & $(x-1)\cdot m_{\alpha}\cdot m_{\alpha^{q+1}}$ & $q^2 - 2q - 1$ & $q+2$ & $0$  & yes
			\\ \hline
			
			$\C_9$ & 
			$(x-1)\cdot m_{\alpha}\cdot m_{\alpha^{q+1}}\cdot m_{\alpha^{q+2}}$ & $q^2 - 2q - 2$ & $q+3$ & $2$ & no
			
	\end{tabular}}
	\caption{The parity-check polynomials (the restrictions as in Table \ref{ToricCodesOnDelPezzoSurfaces})}
	\label{ParityCheckPolynomials}
\end{table}

The output of the code \cite{MagmaToricCodesC3C4} written in the language of the CAS Magma motivates us to formulate

\begin{conjecture}
	The lower bounds from Table {\rm\ref{ToricCodesOnDelPezzoSurfaces}} for the minimum distance $d$ of the codes $\C_3$, $\C_4$ are exact.
\end{conjecture}

\noindent The codes $\C_3$, $\C_4$ for small $q$ are represented in Tables \ref{CqS3T3}, \ref{CqS4T4}. The column LB($d$) (lower bound on $d$ for fixed $q$, $n$, $k$) is rewritten from the Brouwer--Grassl tables \cite{Grassl}. Note that $\C_3$ for $q = 3$ and $\C_4$ for $q = 7$ (cf. \cite{DaskalovGulliver}) have parameters that are the best known at the moment.

\begin{table}[h]
	\quad\parbox{.45\linewidth}{	
		\centering
		\begin{tabular}{c|c|c|c|c}		
			
			$q$ & $n$ & $k$ & $d$ & LB($d$) \\
			\hline
			\hline	
			
			
			$3$ & $13$ & \multirow{6}{*}{$4$} & $7$ & $7$ \\ \cline{1-2}\cline{4-5}
			
			$4$ & $21$ & & $12$ & $14$ \\ \cline{1-2}\cline{4-5}
			
			$5$ & $31$ & & $22$ & $23$ \\ \cline{1-2}\cline{4-5}
			
			$7$ & $57$ & & $45$ & $47$ \\ \cline{1-2}\cline{4-5}
			
			$8$ & $73$ & & $61$ & $62$ \\ \cline{1-2}\cline{4-5}
			
			$9$ & $91$ & & $76$ & $79$	
			
		\end{tabular}
		\caption{The code $\C_3$ for small $q$}
		\label{CqS3T3}	
	}\qquad
	\parbox{.45\linewidth}{		
		\centering
		\begin{tabular}{c|c|c|c|c}
			
			$q$ & $n$ & $k$ & $d$ & LB($d$) \\
			\hline
			\hline
			
			
			$3$ & $10$ & \multirow{6}{*}{$5$} & $4$ & $5$ \\ \cline{1-2}\cline{4-5} 
			
			$4$ & $17$ & & $9$ & $10$ \\ \cline{1-2}\cline{4-5}
			
			$5$ & $26$ & & $16$ & $17$ \\ \cline{1-2}\cline{4-5}
			
			$7$ & $50$ & & $38$ & $38$ \\ \cline{1-2}\cline{4-5}
			
			$8$ & $65$ & & $51$ & $52$ \\ \cline{1-2}\cline{4-5}
			
			$9$ & $82$ & & $66$ & $67$
			
		\end{tabular}
		\caption{The code $\C_4$ for small $q$}
		\label{CqS4T4}
	}	
\end{table}

\begin{remark}
	The codes $\C_3$, $\C_4$ can be naturally generalized, using for any $r \in \N$ the multiple polygons $rPol_3$, $rPol_4$ as well as it is done for $\C_6$ in  {\rm\cite[Theorem 29]{MyArticle2019}}. However, in this case it seems that there are no elegant ways to quite exactly estimate the minimum distance $d$.
\end{remark}

\end{document}